\documentclass[reqno,a4paper]{amsart}
\usepackage[paper=a4paper,left=30mm,right=20mm,top=25mm,bottom=30mm]{geometry}
    \newenvironment{dedication}
        {\vspace{6ex}\begin{quotation}\begin{center}\begin{em}}
        {\par\end{em}\end{center}\end{quotation}}
\usepackage[utf8]{inputenc}
\usepackage[T1]{fontenc}
\usepackage{libertine}
\usepackage[english]{babel}
\usepackage{amssymb}
\usepackage[foot]{amsaddr}

\usepackage{graphicx}
\usepackage{wrapfig}
\usepackage{subfig}
\usepackage{enumerate}
\usepackage[usenames,dvipsnames]{color}
\usepackage{dsfont}
\usepackage{microtype}

\usepackage{etoolbox}

\usepackage{mathtools}
\mathtoolsset{showonlyrefs,showmanualtags}
\usepackage{pstricks,pst-plot,pst-math} 
\usepackage{pstricks-add}

\calclayout

\newtheorem{thm}{Theorem}%[section]

\newtheorem{lem}[thm]{Lemma}

\theoremstyle{definition}
\newtheorem{defn}[thm]{Definition}

\theoremstyle{remark}

\newcommand{\DeclareAutoPairedDelimiter}[3]{%
	\expandafter\DeclarePairedDelimiter\csname Auto\string#1\endcsname{#2}{#3}%
	\begingroup\edef\x{\endgroup
		\noexpand\DeclareRobustCommand{\noexpand#1}{%
			\expandafter\noexpand\csname Auto\string#1\endcsname*}}%
	\x}

\DeclareAutoPairedDelimiter{\abs}{\lvert}{\rvert}
\DeclareAutoPairedDelimiter{\norm}{\lVert}{\rVert}
\DeclareAutoPairedDelimiter{\bra}{(}{ )}
\DeclareAutoPairedDelimiter{\pra}{[}{]}
\DeclareAutoPairedDelimiter{\set}{\{}{\}}
\DeclareAutoPairedDelimiter{\skp}{\langle}{\rangle}

\DeclareMathAlphabet{\mathup}{OT1}{\familydefault}{m}{n}
\newcommand{\dx}[1]{\mathop{}\!\mathup{d} #1}

\DeclareMathOperator{\tr}{tr}

\newcommand{\R}{\mathds{R}}

\usepackage[colorinlistoftodos,prependcaption,textsize=tiny]{todonotes}

\newtoggle{Nebenrechnungen}
\togglefalse{Nebenrechnungen}

\usepackage{hyperref}
\definecolor{darkblue}{rgb}{0,0,0.6}
\hypersetup{
    pdftitle={Characterizing compact coincidence sets in the obstacle problem---a short proof},    % title
    pdfauthor={Simon Eberle},
    colorlinks=true,       % false: boxed links; true: colored links
    linkcolor=darkblue,          % color of internal links (red)
    citecolor=darkblue,        % color of links to bibliography (green)
    filecolor=darkblue,      % color of file links (magenta)
    urlcolor=darkblue           % color of external links (cyan)
}

\title{Characterizing compact coincidence sets in the obstacle problem---a short proof}

\author{Simon Eberle$^1$}
\address{$^1$Faculty of Mathematics, University of Duisburg-Essen, Germany}
\email{simon.eberle@uni-due.de}
\author{Georg S. Weiss$^2$}
\address{$^2$Faculty of Mathematics, University of Duisburg-Essen, Germany}
\email{georg.weiss@uni-due.de}

\let\rho\varrho
\let\phi\varphi
\let\epsilon\varepsilon
\begin{document}
\begin{dedication}
\hspace{4cm}
\vspace*{9cm}{Dedicated to Nina Nikolaevna Uraltseva on the occasion of her 85th birthday.}
\end{dedication}
	\maketitle
\section{Introduction}
The problem of characterizing global solutions of the obstacle problem ---while being crucial to the analysis of the behavior of the free boundary close to singularities--- originated in characterizing null quadrature domains in potential analysis. The first partial result in dimension $n=3$ is due to P. Dive (\cite{dive}). H. Lewy (\cite{Lewy79}), too, arrived at ellipsoids in $n=3$.
In two dimensions a complete characterization of null quadrature domains was proved by Makoto Sakai (\cite{Sakai}) by means of complex analysis: only half planes, ellipsoids and paraboloids are possible.
E. DiBenedetto- A. Friedman (\cite[Theorem 5.1]{DiBenedettoFriedman}) in higher dimensions showed ---drawing on the result by P. Dive--- that any bounded domain $K$ with non-empty interior such that the gravity force produced by the homoeoid $\lambda K \setminus K$ is zero in a neighborhood of the origin then $K$ is an ellipsoid, that is a set of the form
$\set {x \in \R^n : x^T A x \leq 1}$, where $A$ is a positive definite, symmetric matrix $A \in \R^{n \times n}$. E. DiBenedetto- A. Friedman applied this result to classical solutions of the Hele-Shaw flow (which includes classical solutions of the one-phase obstacle problem) where they prove that provided that the coincidence set $\{u=0\}$ of the solution $u$ is bounded with non-empty interior and that {\em $\{u=0\}$ is symmetric with respect to each hyperplane $\{ x_j=0\}$}, 
$\{u=0\}$ is an ellipsoid. Finally M. Sakai-A. Friedman (\cite{FriedmanSakai}) removed the unnecessary symmetry assumption in the null quadrature domain setting.

Motivated by the almost completely open problem of characterizing unbounded coincidence sets in higher dimensions (compare the conjectures in \cite{Karp}),
we give in this note a concise and easy-to-extend proof of the fact that
if coincidence set $\set {u=0}$ is bounded with nonempty interior then it is an ellipsoid. Our proof is based on the idea in \cite{dive} of touching the unknown set $\set {u=0}$ with a known ellipsoid. However we stick during the whole proof to the obstacle problem setting, that is we work with the solution $u$ avoiding homoeoids and null quadrature domains which will make our result rather short. We will neither need symmetry of $\set {u=0}$ at any stage of the proof nor any regularity assumption of $\set {u=0}$. 

\section*{Acknowledgments}
The authors thank the Hausdorff Research Institute for Mathematics of the university of Bonn for its hospitality, Herbert Koch for discussions and Luis Caffarelli for pointing out reference \cite{Lewy79} to the authors.
 
\section{Main Theorem and Proof}
Let $u$ be a nonnegative weak solution of
\begin{equation} \label{PDE}
	\Delta u = \chi_{\set {u>0}} \quad \text{ in } \R^n.
\end{equation}
It is a known fact that the second derivatives are globally bounded, i.e.
\begin{equation}\label{bounded_second_derivative}
\norm {D^2u}_{L^\infty(\R^n)} < +\infty.
\end{equation}
and (see for example \cite {Monneau2003} or \cite{PetrosyanShahgholianUraltseva_book}) that 
 \begin{equation} \label{PDE_asymptotics}
 	\frac{u(rx)}{r^2} \to x^T Q x \quad \text{ in } C^{1,\alpha}_{loc}\cap
W^{2,p}_{loc}\text{ as } r \to \infty,
 \end{equation}
 where $Q  \in \R^{n \times n}$ is positive definite, symmetric and $\tr(Q) =  \frac{1}{2}$.
Without loss of generality we assume that the coordinate system is rotated such that $Q$ is a diagonal matrix. 
 
\begin{defn}
	We define an ellipsoid as a set $\set {x \in \R^n : x^T A x \leq 1}$, where $A$ is a positive definite, symmetric matrix $A \in \R^{n \times n}$.
\end{defn}
\begin{thm}[Main Theorem]\label{theorem:main_theorem}
	Let $n\ge 3$ and $u$ be as above. If the coincidence set $\set {u=0}$ is bounded with nonempty interior then it is an ellipsoid.
\end{thm}
The idea underlying the proof is extremely simple. We will touch the coincidence set $\set {u=0}$ with a suitable ellipsoid and apply a strong comparison principle to the respective solutions. Of course not any ellipsoid will do. We will need the following Lemma relating ellipsoids to solutions of the obstacle problem.
\begin{lem}[Existence of ellipsoid solutions] \label{lemma:existence_of_ellipsoidal_solutions}
For any polynomial $p(x) = x^T Q x$, where $Q \in \R^{n \times n}$ is diagonal, positive definite and $\operatorname{tr}(Q) = \frac{1}{2}$, there is an ellipsoid $E$ (symmetric with respect to $\set {x_j = 0}$ for all $j \in \set {1, \dots N}$) and a nonnegative solution  of the obstacle problem $u^E$ such that
\begin{align}
	\Delta u^E = \chi_{\set {u^E >0}} ~ \text{ in } \R^n \quad , \quad \set {u^E=0} = E \quad \text{ and } \quad \frac{u^E(rx)}{r^2} \to p(x) ~ \text{ as } r \to \infty.
\end{align}
\end{lem}
 \begin{proof}[Proof of Lemma \ref{lemma:existence_of_ellipsoidal_solutions}]
 	From \cite[see (5.4) therein]{DiBenedettoFriedman} we know that for any polynomial $p(x) = x^T Q x$ there is an ellipsoid $E:= \{x \in \R^n : x^T A x \leq 1\}$ ($A \in \R^{n \times n}$ positive definite, diagonal and symmetric, as in the statement of the Lemma.) such that its Newton-Potential
 	\begin{align}
 		u_E^{NP}(x) :=c(n) \int \limits_E \frac{1}{\abs {x-y}^{n-2}} \dx{y} = u_E^{NP}(0) - p(x) \quad \text{ in } E.
 	\end{align}
 	Here $c(n)$ is given by $c(n):= \frac{1}{n (n-2) \abs{B_1}}$.
 	We now define the solution $u^E$ by
 	\begin{align}
 		u^E(x) := p(x) - u_E^{NP}(0)+u_E^{NP}(x).
 	\end{align}
 	A direct computation shows that
 	\begin{align}
 		\Delta u^E = 1-\chi_E = \chi_{\R^n \setminus E} \quad , \quad u^E = 0 \text{ in } E      \quad \text{ and } \quad  \frac{u^E(rx)}{r^2} \to p(x) ~ \text{ as } r \to \infty ,	\end{align}
 		where we have used that $u_E^{NP}(x) \to 0$ as $\abs{x} \to \infty$.
 	So all we need to check is that $u^E>0$ in $\R^n \setminus E$. From \cite[Theorem II]{CaffarelliKarpShahgholian_Pompeiu} we infer that $u^E$ is nonnegative in $\R^n$ and that the coincidence set $\set {u^E = 0}$ is convex. Hence if there was $x_0 \not \in E$ such that $u^E(x_0)=0$, then this would imply that $\operatorname{conv}(\set{x_0} \cup E) \subset \set {u^E = 0}$ but this is impossible since $\Delta u^E =1$ in $\R^n \setminus E$.
 \end{proof}

{\em Proof of the Main Theorem:}\\
{\bf Step 1:} \emph{The Newton-potential solution}\\
In order to get a better understanding of the higher order asymptotics of the solution as $\abs{x} \to \infty$, we decompose it into a polynomial and the Newton-potential solution.

First, we modify the solution, such that the Laplacian is supported on a bounded domain. Setting
\begin{align}
	v(x):= u(x) - x^T Q x \quad \text{ in } \R^n,
\end{align}
$v$ solves
\begin{align}
	\Delta v = -\chi_{\set {u=0}}.
\end{align}
Let us denote $K:= \set {u=0}$ which is \emph{compact} by the assumption in Theorem \ref{theorem:main_theorem} and \emph{convex} because of \eqref{bounded_second_derivative} (see e.g. \cite[Theorem 5.1]{PetrosyanShahgholianUraltseva_book}). From here on we assume that the coordinate system is translated in such a way that
\begin{align}
	\int \limits_K \frac{y}{\abs{y}^n} \dx{y} = 0 \quad \text{ and } 0 \in \operatorname{int} K.
\end{align}
This is possible because $K$ is bounded and convex with nonempty interior. For details see Appendix \ref{appendix:choice_of_zero_is_possible}.

The Newton-potential solution (for $n \geq 3$) given by
\begin{align}
	v^{NP}(x) := c(n) \int \limits_K \frac{1}{\abs {x-y}^{n-2}} \dx{y} \quad , \text{ where } c(n) := \frac{1}{n (n-2) \abs{B_1}} >0
\end{align}
is a strong solution in $W^{2,p}_{loc}$ of
\begin{align}
	\Delta v^{NP} = -\chi_K.
\end{align}
Let us note that $\Delta (v-v^{NP})\equiv 0$, i.e the difference is harmonic. Since $\frac{(v-v^{NP})(rx)}{r^2} \to 0$ as $r \to \infty$ uniformly on $\partial B_1$ the difference must by a Liouville-type argument be a harmonic polynomial of degree at most one, in the following denoted by $p$.

Recall that $0 \in \operatorname{int} K$. This allows us to deduce that
\begin{align}
\nabla p(0) &= \nabla v(0) - \nabla v^{NP}(0) = 0 - c(n) \int \limits_K \frac{y}{\abs{y}^n} \dx{y} =0 , \\
p(0) &= v(0)-v^{NP}(0) = 0- v^{NP}(0) <0.
\end{align}
We infer that $p \equiv p(0)<0$ and that $v \equiv v^{NP} +p(0) = v^{NP}- v^{NP}(0)$.

{\bf Step 2:} \emph{The comparison function}\\
The idea in the following is to touch $K$ from the outside with an ellipsoid $E$ satisfying $$\int \limits_E \frac{y}{\abs{y}^n} \dx{y} =\int \limits_K \frac{y}{\abs{y}^n} \dx{y}, $$ and to compare $u$ with the respective solution $u^E$.
To do so, let us choose $E \subset \R^n$ as in Lemma \ref{lemma:existence_of_ellipsoidal_solutions} where we set $p$ in the lemma to be the blow-down of $u$ as defined in \eqref{PDE_asymptotics}.
The symmetry of $E$ yields that
\begin{align}
	\int \limits_E \frac{y}{\abs {y}^n} =0.
\end{align}
 Furthermore we define the family of ellipsoids
\begin{align}
(E_r)_{r >0} \quad , \quad E_r := \frac{1}{r} E 	
\end{align}
and the respective rescalings
	\begin{align}
		U_r(x) := \frac{u^E(rx)}{r^2} \quad \text{ in } \R^n.
	\end{align}
Note that for all $r>0$ we have that $U_r$ is a nonnegative solution of the obstacle problem \eqref{PDE} with coincidence set $\set {U_r = 0} = E_r$ and blow-down $p(x) = x^T Q x$.

Using the compactness of $K$ there is $r_1 > 0$ such that $K \subset \operatorname{int} E_{r_1}$.

As for $u$ we modify $U_r$ to
\begin{align}
	V_r(x) := U_r(x) - x^T Q x \quad \text{ for all } x \in \R^n.
\end{align}
It follows that
\begin{align}
	\Delta V_r = - \chi_{\set {U_r = 0}} = -\chi_{E_r}.
\end{align}
As before we define the Newton-potential solution
\begin{align}
	V_r^{NP}(x) := c(n) \int \limits_{E_r} \frac{1}{\abs {x-y}^{n-2}} \dx{y}.
\end{align}
Since $E_r$ is symmetric with respect to all planes $\set {x_i = 0}$ for $i \in \set {1, \dots, n}$ it follows that
\begin{align}
	\nabla V_r^{NP}(0) = c(n) \int \limits_{E_r} \frac{y}{\abs{y}^n} \dx{y} =0.
\end{align}
Since $\Delta (V_r - V_r^{NP}) \equiv 0$, the difference $V_r - V_r^{NP}$ is harmonic in the entire space, and from $\frac{(V_r - V_r^{NP})(sx)}{s^2} \to 0$ as $s \to \infty$
uniformly in $\partial B_1$ we infer again from a Liouville-type argument that the difference must be polynomial of degree at most one and we denote it by $p_r$.  
Calculations similar to the calculations above show that $p_r$ must be of degree zero:
\begin{align}
	\nabla p_r(0) &= \nabla V_r(0) - \nabla V_r^{NP}(0) = 0 - c(n) \int \limits_{E_r} \frac{y}{\abs{y}^n} \dx{y} =0 \\
	p_r(0) &= V_r(0) -V_r^{NP}(0) = -V_r^{NP}(0) <0.
\end{align}
It follows that $V_r \equiv V_r^{NP} - V_r^{NP}(0)$ for all $r>0$.
Now we are able to prove the following comparison Lemma.

{\bf Step 3:} \emph{Comparison principle} \\
For all  $r >0$ such that $K \subset E_r$ and $\abs {E_r \setminus K} \neq 0$,
\begin{align}
	u \geq U_r \quad \text{ in } \R^n \quad \text{ and } \quad u > U_r \quad  \text{ in } \R^n \setminus E_r.
	\end{align}
For a proof, first note
that for all $x \in \R^n$
\begin{align} \label{Newton_potential_comparison}
	V_r^{NP}(x) = c(n) \int \limits_{E_r} \frac{1}{\abs {x-y}^{n-2}} \dx{y} > c(n) \int \limits_{K} \frac{1}{\abs {x-y}^{n-2}} \dx{y} = v^{NP}(x).
\end{align}
Let us now apply \eqref{Newton_potential_comparison} to the difference
\begin{align}
	(u-U_r)(x) = (v-V_r)(x) = v^{NP}(x) -V_r^{NP}(x) + V_r^{NP}(0)- v^{NP}(0) \to   V_r^{NP}(0)- v^{NP}(0) >0 \quad \text{ uniformly as } \abs{x} \to \infty.
\end{align}
Hence there is $R>0$ such that 
\begin{align} \label{u_and_u_r_are_ordered_outside_B_R}
	u > U_r \quad \text{ in } \R^n \setminus B_R(0).
\end{align}
As $u$ and $U_r$ solve the same PDE \eqref{PDE} and we have a comparison principle for this nonlinear PDE (for details see Appendix \ref{appendix:comparison_principle}) we infer that
\begin{align}
u \geq U_r \quad \text{ in } \R^n,
\end{align}
and furthermore
\begin{align}
	u > U_r \quad \text{ in } \R^n \setminus E_r .
\end{align}
The strict inequality holds because $u-U_r$ is harmonic in $\R^n \setminus E_r$, and in case of the two graphs touching in $\R^n \setminus E_r$ the strong maximum principle would  yield that $u-U_r \equiv 0$ in $\R^n \setminus E_r$, contradicting \eqref{u_and_u_r_are_ordered_outside_B_R}.

{\bf Step 4:} \emph{Applying Hopf's principle to finish the proof}\\

\begin{figure}[h!]
	\frame{
\psset{xunit=1.0cm,yunit=1.0cm}
\begin{pspicture*}
%[showgrid=true]
(-5,-0.5)(5,4.9)
%\psdot[linecolor=red](0,0)
\psline(0,4.0)(1.7,3.3)(1.4,1.4)(-0.5,1.3)(-1.2,3)(0,4.0)
\psdot[dotscale=1.1](0,4)
\rput(0.9,2.0){$K$}
\rput(0,4.3){$x_0$}
\psellipse(0.0,2)(4.5,2)
\rput(3,2){$E_{r_0}$}
\end{pspicture*}
}
\caption{Touching the obstacle from outside.}
\label{fig:touch_from_outside}
\end{figure}
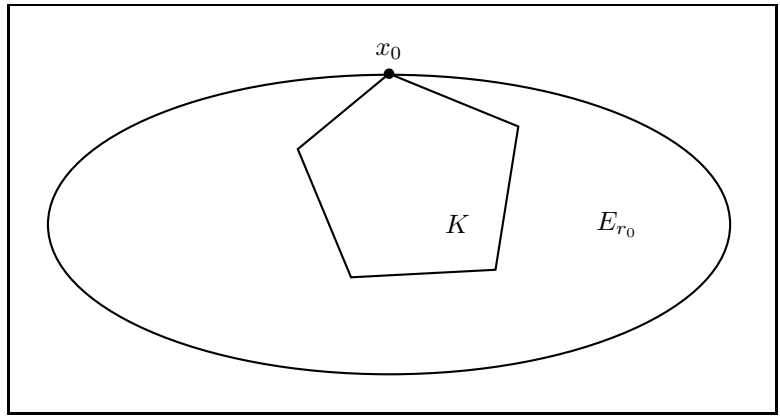

Let us now increase $r>0$ from $r=r_1$ to $r_0 >r_1$ such that the boundaries of $E_{r_0}$ and $K$ touch for the first time (see Figure \ref{fig:touch_from_outside}), i.e.
\begin{align}
	\partial E_{r_0} \cap \partial K \neq \emptyset \text{ and } 	\partial E_{r} \cap \partial K = \emptyset \text{ for all } r<r_0. 
\end{align}
Let $x_0 \in 	\partial E_{r_0} \cap \partial K$ be a touching point.
Then either $K = E_{r_0}$ and Theorem \ref{theorem:main_theorem} is proved, or $K \neq E_{r_0}$.
The latter would imply that $\abs{E_{r_0} \setminus K} > 0$. In order to see this assume that $A, B \subset \R^N$ are two convex, compact sets with non-empty interior such that $A \subset B$ and $\abs{B \setminus A} =0$. Then since $(\operatorname{int} B) \setminus A$ is open, $\abs{B \setminus A} =0$ implies that $(\operatorname{int} B) \setminus A = \emptyset$. Thus $\operatorname{int} B \subset A$. 
If there is $\bar{x} \in \partial B \setminus A$, then since $B$ is convex, $\operatorname{conv}(\set {\bar{x}} \cup A) \subset B$ and $\operatorname{conv}(\set {\bar{x}} \cup A) \setminus A$ has non-empty interior, as $\operatorname{int} B \neq \emptyset$, $\operatorname{int} B \subset A$ and $\operatorname{dist}(\bar{x}, A) >0$. It follows that $\abs{B \setminus A} >0$, a contradiction.

  This allows us to apply Step 3 with $r=r_0$ and to obtain that
\begin{align}
	u > U_{r_0} \quad \text{ in } \R^n \setminus E_{r_0}.
\end{align}
Furthermore
\begin{align}
	\Delta (u- U_{r_0}) = 0 \quad \text{ in } \R^n \setminus E_{r_0}.
\end{align}
Since the boundary of the ellipsoid is smooth, there is an open ball $B \subset \R^n \setminus E_{r_0}$ such that
\begin{align}
	\bar{B} \cap E_{r_0} = \set{x_0} .
\end{align}
From the classical Hopf principle we infer that 
\begin{align}
	\frac{\partial(u-U_{r_0})}{\partial \nu}(x_0) <0,
\end{align}
where $\nu$ is the outer unit normal on $\partial B$ at $x_0$. But this is impossible since $x_0$ is a free boundary point of both $u$ and $U_{r_0}$, implying that
\begin{align}
	\nabla u(x_0) = 0 = \nabla U_{r_0}(x_0).
\end{align}
Hence the assumption $K \neq E_{r_0}$ must have been wrong and Theorem \ref{theorem:main_theorem} is proved.\qed

\vfill

\appendix

\section{The weighted center of gravity of $K$ can be chosen to be zero by a translation} \label{appendix:choice_of_zero_is_possible}
We will show that there is $x^0\in K$ such that
	\begin{align}
	\int \limits_{K - x^0} \frac{y}{ \abs {y}^n } \dx{y} = 0.
	\end{align}
	\begin{enumerate}[1.]
	\item First we show that there is $x^0 \in \R^n$ such that 
	\begin{align}\label{center}
	\int \limits_{K - x^0} \frac{y}{ \abs {y}^n } \dx{y} = 0.
	\end{align}
	To do so we choose $R>0$ such that $K \subset \bar B_R$ and for all $\epsilon >0$ define the continuous operator $T^\epsilon: \bar B_R \to \R^n$ 
	\begin{align}
	T^\epsilon(x) := \frac{  \int \limits_K \frac{y}{\abs {y-x}^{n-\epsilon}}  \dx{y}  }{\int \limits_K \frac{1}{ \abs {y-x}^{n-\epsilon} } \dx{y}} .
	\end{align}
	(Note that we have only employed the regularization with $\epsilon$ in order to ensure integrability of all terms involved.)
	$T^\epsilon$ is a self-map because
	\begin{align}
	\abs {T^\epsilon(x)} \leq \frac{R \int \limits_K \frac{1}{ \abs {y-x}^{n-\epsilon} } \dx{y} }{\int \limits_K \frac{1}{ \abs {y-x}^{n-\epsilon} }  \dx{y}} = R.
	\end{align}
	Now Brouwer's fixed point theorem yields that for all $\epsilon >0$ the continuous self-map $T^\epsilon: \bar B_R \to \bar B_R$ has a fixed point $x^\epsilon \in \bar B_R$. 
	Since for all $\epsilon >0$ we know that $x^\epsilon \in \bar B_R$,
	\begin{align}
		x^{\epsilon_m} \to x^0 \in \bar B_R \quad \text{ as } m \to \infty.
	\end{align}
Since
$$\left|\frac{y-x^{\epsilon_m}}{\abs{y-x^{\epsilon_m}}^{n-{\epsilon_m}}}\right|
\le \chi_K\left(\abs{y-x^{\epsilon_m}}+\left|\frac{y-x^{\epsilon_m}}{\abs{y-x^{\epsilon_m}}^{n}}\right|\right)$$
where the right-hand side converges in $L^1$,
Lebesgue's (generalized) convergence theorem implies that
$$\int\limits_K \frac{y-x^{\epsilon_m}}{\abs{y-x^{\epsilon_m}}^{n-{\epsilon_m}}}\dx{y}
\to \int \limits_K \frac{y-x^0}{\abs{y-x^0}^{n}} \dx{y}$$
as $m\to\infty$.\\
	(Here we used the boundedness of $K$.)
	\item It remains to show that $x^0\in K$.
Assume that $x^0 \not \in K$. 
By the convexity of $K$ there is a hyperplane separating $x^0$ and $K$.
Let $\nu$ be a unit normal on that hyperplane. By a translation and rotation we may assume that $x^0 = 0$, $\nu = e_1$ and $K \subset \set {x_1 >0}$ or $K \subset \set {x_1 <0}$, implying that
\begin{align}
	\int \limits_{K} \frac{y_1}{ \abs {y}^n } \dx{y} \ne 0,
	\end{align}
contradicting \eqref{center}.
\end{enumerate}

\section{Comparison principle for the nonlinear PDE} \label{appendix:comparison_principle}
The following comparison principle for solutions of the obstacle problem is well-known and stated only for the sake of completeness.\\
Let $\Omega \subset \R^n$ be a bounded domain with Lipschitz-boundary and let $u$ and $v$ be weak solutions of \eqref{PDE}, i.e. $u,v \in W^{1,2}(\Omega)$ such that for all $\phi \in W^{1,2}_0(\Omega)$ it holds that
\begin{align} \label{weak_formulation_comparison}
	-\int \limits_\Omega \nabla u \cdot \nabla \phi = \int \limits_\Omega \chi_{\set {u>0}} \phi \quad \text{ and  } \quad  	-\int \limits_\Omega \nabla v \cdot \nabla \phi = \int \limits_\Omega \chi_{\set {v>0}} \phi.
\end{align}
Assume furthermore
\begin{align} \label{weak_comparison_principle_boundary_condition}
	v \leq u \quad \text{ on } \partial \Omega
\end{align}
in the sense of traces. Then, testing the weak formulation \eqref{weak_formulation_comparison} with the (admissible) test function $(v-u)^+ := \max \set {v-u,0}$ and subtracting the equations we obtain
\begin{align}
	- \int \limits_\Omega \abs {\nabla (v-u)^+}^2 = \int \limits_\Omega \bra {\chi_{\set{v>0}}-\chi_{\set {u>0}} } (v-u)^+ \geq 0.
\end{align} 
Hence 
\begin{align}
 \nabla (v-u)^+ \equiv 0 \quad \text{ a.e. in } \Omega.
\end{align}
This implies that
\begin{align}
	(v-u)^+ \equiv \operatorname{constant} \quad \text{ a.e. in } \Omega,
\end{align}
and from \eqref{weak_comparison_principle_boundary_condition} we infer that
\begin{align}
	(v-u)^+ \equiv 0  \quad \text{ a.e. in }\Omega,
\end{align}
which finishes the proof.

\bibliographystyle{abbrv}
\bibliography{ellipsiod_arXiv}

\end{document}